\documentclass{amsart}
\usepackage{hyperref}
\hypersetup{bookmarksnumbered} % Adds numbers to the section titles in the PDF table of contents
\usepackage{amssymb,amsmath,amsthm}
\usepackage{stackengine} % for stacked text
\usepackage{graphicx}
\usepackage{ulem}
\usepackage[subrefformat=parens,labelformat=parens]{subfig}

\input xy
\xyoption{all}

\usepackage{color}

\makeatletter

\let\c@table\c@figure
\makeatother

\newtheorem{theorem}{Theorem}
\newtheorem{proposition}[theorem]{Proposition}
\newtheorem{lemma}[theorem]{Lemma}

\theoremstyle{definition}

\theoremstyle{remark}
\newtheorem{remark}[theorem]{Remark}

\newcommand{\newword}[1]{\textbf{#1}}

\newcommand{\Z}{\mathbb{Z}}
\newcommand{\N}{\mathbb{N}}

\newcommand{\C}{\mathfrak{C}}
\newcommand{\T}{\mathcal{T}}

\DeclareMathOperator{\Homeo}{Homeo}

\DeclareMathOperator{\Stab}{Stab}

\newcommand{\Aut}{\mathrm{Aut}}

\newcommand{\Fix}{\mathrm{Fix}}

\title{Boone--Higman Embeddings for Contracting Self-Similar Groups}

\author{James Belk}
\address{School of Mathematics \& Statistics, University of Glasgow, Glasgow, G12~8QQ, Scotland.}
\email{\href{mailto:jim.belk@glasgow.ac.uk}{jim.belk@glasgow.ac.uk}}

\author{Francesco Matucci}
\address{Dipartimento di Matematica e Applicazioni, Universit\`{a} degli Studi di Milano--Bicocca, Milan 20125, Italy.}
\email{\href{mailto:francesco.matucci@unimib.it}{francesco.matucci@unimib.it}}
\thanks{The third author is a member of the Gruppo Nazionale per le Strutture Algebriche, Geometriche e le loro Applicazioni (GNSAGA) of the Istituto Nazionale di Alta Matematica (INdAM) and gratefully acknowledges the support of the 
of the Universit\`a degli Studi di Milano--Bicocca
(FA project 2021-ATE-0033 ``Strutture Algebriche'').
}

\date{}  % Removes displayed date

\begin{document}

\maketitle

\begin{abstract}We give a short proof that every contracting self-similar group embeds into a finitely presented simple group.   In particular, any contracting self-similar group embeds into the corresponding R\"over--Nekrashevych group, and this in turn embeds into one of the twisted Brin--Thompson groups introduced by the first author and Matthew Zaremsky.  The proof here is a simplification of a more general argument given by the authors, Collin Bleak, and Matthew Zaremsky for contracting rational similarity groups.
\end{abstract}

\section{Introduction}

In 1973, William Boone and Graham Higman conjectured that a finitely generated group has solvable word problem if and only if it embeds into a finitely presented simple group~\cite{Boone,BoHi}.  Elizabeth Scott proved in 1984 that $\mathrm{GL}_n(\Z)$ embeds into a finitely presented simple group for all $n\geq 2$~\cite{Scott2}, and recent work has demonstrated Boone--Higman embeddings for all countable abelian groups~\cite{BHM2} and for all hyperbolic groups~\cite{BBMZ}.  See \cite{BBMZSurvey} for a survey of related results.

In 1999, Claas R\"over described an embedding of Grigorchuk's group into a finitely presented simple group~\cite{Rov}.  Specifically, the action of Grigorchuk's group $\mathcal{G}$ on the infinite rooted binary tree induces an action of $\mathcal{G}$ on the boundary Cantor set $\C_2=\{0,1\}^\omega$, and R\"over proved that the group $V\mathcal{G}$ of homeomorphisms of $\C_2$ generated by $\mathcal{G}$ and Thompson's group $V$ is finitely presented and simple.  Note that Grigorchuk's group $\mathcal{G}$ is not itself finitely presented~\cite{Grig}, which makes this result particularly surprising. 

Grigorchuk's group is the archetypical example of a self-similar group, as defined by Volodomyr Nekrashevych in~\cite{NekBook}.  Unfortunately, R\"over's proof does not extend easily to other self-similar groups.  If $G$ is a self-similar group acting on an infinite rooted $d$-ary tree $\mathcal{T}_d$, then $G$ can be viewed as a group of homeomorphisms of the Cantor set $\mathfrak{C}_d=\{0,\ldots,d-1\}^\omega$, and one can construct a R\"over--Nekrashevych group $V_dG$ analogous to R\"over's group~\cite{Scott,Nek}.  Unfortunately, $V_dG$ is neither finitely presented nor simple in general, though in some cases this construction gives new finitely presented simple groups.  For example, Rachel Skipper, Stefan Witzel, and Matthew Zaremsky used R\"over--Nekrashevych groups to construct the first family of simple groups with arbitrary finiteness lengths~\cite{SWZ}.

In 2017, Nekrashevych proved that $V_d G$ is indeed finitely presented whenever the self-similar group $G$ is contracting~\cite{Nek}.   Contracting groups are arguably the most important class of self-similar groups, and play a central role in the theory of iterated monodromy groups and limit spaces developed by Nekrashevych~\cite{NekBook}.  In~\cite{BBMZ}, the authors together with Collin Bleak and Matthew Zaremsky generalized Nekrashevych's result to arbitrary full, contracting rational similarity groups, which act on subshifts of finite type by asynchronous automata.  It was also proven in~\cite{BBMZ} that any contracting rational similarity group embeds into a finitely presented simple group, yielding Boone--Higman embeddings for contracting self-similar groups as well as all hyperbolic groups.

In this note we give a short proof that contracting self-similar groups admit Boone--Higman embeddings.  We use only Nekrashevych's finite presentation result, and we do not require the more general theory of rational similarity groups.  Our main theorem is the following.

\begin{theorem}\label{thm:MainTheorem}
Every contracting self-similar group embeds into a finitely presented simple group.
\end{theorem}

The strategy is to prove that the associated R\"over--Nekrashevych group embeds into one of the twisted Brin--Thompson groups introduced by the first author and Matthew Zaremsky in 2022~\cite{BZ}.  Given any faithful action of a group $H$ on a set~$S$, there is an associated twisted Brin--Thompson group $SV_H$ which is simple and has $H$ as a subgroup. 
 In a recent article~\cite{ZaremskyTaste}, Zaremsky proved that $SV_H$ is finitely presented as long as the following conditions are satisfied:
\begin{enumerate}
\item $H$ is finitely presented,\smallskip
\item $H$ acts oligomorphically on $S$, i.e.\ the induced action of $H$ on $S^n$ has finitely many orbits for all $n\geq 1$.\smallskip
\item The stabilizer of any finite subset of $S$ is finitely generated.
\end{enumerate}
Here we prove that if $G$ is a contracting self-similar group, then the action of the R\"over--Nekrashevych group $H = V_d G$ on any orbit $S$ of rational points in $\C_d$ satisfies conditions (2) and (3) above.  Combining this with Nekrashevych's finite presentation result, it follows that the twisted Brin--Thompson group $SV_H$ is a finitely presented simple group that contains $H$ and hence $G$.

Incidentally, Scott proved in \cite{Scott} that if a self-similar group $G\leq \Aut(\T_d)$ is itself finitely presented, then so is the associated R\"over--Nekrashevych group~$V_dG$.  Zaremsky has recently proven that all such groups embed into finitely presented simple groups~\cite{ZaremskyFP}, and his proof also works for contracting self-similar groups that are finitely generated.

\subsection*{Acknowledgments}

The authors would like to thank Collin Bleak, James Hyde and Matthew Zaremsky for many helpful discussions.

\section{The Proof}

For $d\geq 2$, let $X_d=\{0,\ldots,d-1\}$, and let  $\mathcal{T}_d$ be the infinite, rooted $d$-ary tree whose nodes are elements of $X_d^*$.  If $g\in \Aut(\T_d)$ and $\alpha\in X_d^*$, there exists a unique automorphism $g|_{\alpha}\in \Aut(\T_d)$ such that
\[
g(\alpha\cdot\beta) = g(\alpha)\cdot g|_{\alpha}(\beta)
\]
for every $\beta\in X_d^*$, where $\cdot$ denotes concatenation.  This automorphism $g|_\alpha$ is known as the \newword{local action} (or ``restriction'') of $g$ at $\alpha$.  A group $G\leq \Aut(\T_d)$ is \newword{self-similar} if $g|_\alpha\in G$ for all $g\in G$ and $\alpha\in X_d^*$.  A self-similar group $G$ is \newword{contracting} if there exists a finite set $\mathcal{N}\subseteq G$, known as the \newword{nucleus} of $G$, such that for every $g\in G$ the set
\[
\{\alpha \in X_d^* \mid g|_\alpha\notin \mathcal{N}\}
\]
is finite.  Equivalently, for every $g\in G$ there exists an $M\geq 0$ so that $g|_\alpha\in\mathcal{N}$ for all words $\alpha\in X_d^*$ of length $M$ or greater.

Now let $\C_d$ denote the Cantor space $\C_d=X_d^\omega$, and note that $\Aut(\T_d)$ acts faithfully on $\C_d$ by homeomorphisms.  For each $\alpha\in X_d^*$, the corresponding \newword{cone} in $\C_d$ is the set $\alpha\C_d$ of all infinite $d$-ary sequences that have $\alpha$ as a prefix.  These cones form a basis for the topology on~$\C_d$.  Indeed, every open set is a disjoint union of cones, and a set is clopen if and only if it is a finite disjoint union of cones.

If $h\in \Homeo(\C_d)$ maps a cone $\alpha\C_d$ to a cone $\beta\C_d$, the corresponding \newword{local action} (or ``restriction'') of $h$ is the homeomorphism $h|_\alpha$ of $\C_d$ defined by
\[
h(\alpha\cdot \psi)=\beta\cdot h|_\alpha(\psi)
\]
for all $\psi\in X_d^\omega$.  This agrees with the earlier definition of a local action in the case where $h\in \Aut(\T_d)$, and has the property that $h|_{\alpha\beta} = (h|_\alpha)|_\beta$ whenever $h|_\alpha$ is defined. 
 If $G\leq \Aut(\T_d)$ is a self-similar group, a homeomorphism $h\in \Homeo(\C_d)$ lies in the \newword{R\"over--Nekrashevych group $\boldsymbol{V_dG}$} if for every point $p\in \C_d$ there exists a cone $\alpha\C_d$ containing $p$ such that $h(\alpha\C_d)$ is a cone and $h|_\alpha\in G$. We shall refer to such a cone $\alpha\C_d$ as a \newword{regular cone} for~$h$.

We begin by verifying condition (2) given in the introduction.

\begin{lemma}\label{lem:MakeElements}
Let $G\leq \Aut(\T_d)$, let $g_1,\ldots,g_n\in G$, and let  $\alpha_1\C_d,\ldots,\alpha_n\C_d$ and $\beta_1\C_d,\ldots,\beta_n\C_d$ be two collections of cones in $\C_d$, where each collection is pairwise disjoint and does not cover~$\C_d$. Then there exists an $h\in V_d G$ such that $h(\alpha_i\C_d)=\beta_i\C_d$ for each~$i$, with $h|_{\alpha_i}= g_i$.
\end{lemma}
\begin{proof}
Note that any partition of $\C_d$ onto cones $k$ cones must satisfy \[
k\equiv 1\pmod{d-1}.
\]
Since $\bigcup_{i=1}^n \alpha_i\C_d$ is a clopen set, its complement is clopen, so we can find a partition $\gamma_1\C_d,\ldots,\gamma_r\C_d$ of the complement into cones.  Similarly, we can find a partition $\delta_1\C_d,\ldots,\delta_s\C_d$ of the complement of $\bigcup_{i=1}^n \beta_i\C_d$ into cones. Then $n+r\equiv n+s\equiv 1\;(\mathrm{mod}\;d-1)$, so $r\equiv s\;(\mathrm{mod}\;d-1)$.  Refining these partitions if necessary, we may therefore assume that $r=s$.  Then the element $h\in V_dG$ which maps each $\alpha_i\C_d$ to $\beta_i\C_d$ with $h|_{\alpha_i}=g_i$ and maps each $\gamma_i\C_d$ to $\delta_i\C_d$ with $h|_{\gamma_i}=\mathrm{id}$ has the desired properties.
\end{proof}

\begin{proposition}\label{thm:oligo}
If $G\leq\Aut(\T_d)$ is a self-similar group, then $V_d G$ acts oligomorphically on each of its orbits in~$\mathfrak{C}_d$.
\end{proposition}
\begin{proof}Let $S$ be such an orbit.  It suffices to prove that for all $n\geq 1$ the group $V_dG$ acts transitively on ordered $n$-tuples of distinct elements of~$S$. %\fra{Why is this sufficient? I guess because if we have two $(n+1)$-tuples, we find $a,b \in G$ so that $a(p_1)=q_1$ and $b(p_i)=q_i$for every $i =2,\ldots,n+1$ and then build a single guy $g$ restricting to both $a$ and $b$?}

Let $(p_1,\ldots,p_n)$ and $(q_1,\ldots,q_n)$ be tuples of distinct points in~$S$.  Since $p_i$ and $q_i$ lie in the same orbit, there exists for each $i$ an $h_i\in V_dG$ that maps $p_i$ to $q_i$.  Then each $h_i$ has a regular cone $\alpha_i\C_d$ that contains $p_i$. 
 Extending the $\alpha_i$'s if necessary, we may assume that the cones $\alpha_i\C_d$ are pairwise disjoint and do not cover $\C_d$, and that their images $h_i(\alpha_i\C_d)$ have these same two properties.  Then by Lemma~\ref{lem:MakeElements}, there exists an element $h\in V_dG$ such that $h(\alpha_i\C_d)=h_i(\alpha_i\C_d)$ for each $i$ with $h|_{\alpha_i}=h_i|_{\alpha_i}$.  Then $h(p_i)=h_i(p_i)=q_i$ for each~$i$, as desired.
\end{proof}

All that remains is condition (3) in the introduction. 
We will only be able to prove this condition under the hypothesis that $G$ is contracting and $S$ is the orbit of a rational point under $V_dG$, where a point $p\in \C_d$ is \newword{rational} if it has the form $\alpha\overline{\beta}$ for some $\alpha\in X_d^*$ and $\beta\in X_d^+$.

The proof of condition (3) requires a preliminary definition.  If $H\leq \Homeo(\C_d)$ and $p\in \C_d$, the \newword{group of germs} of $H$ at $p$ is the quotient
\[
(H)_p = \mathrm{Stab}_H(p) \bigr/ \Fix_H^0(p)
\]
where $\Fix_H^0(p)$ is the subgroup of $H$ consisting of elements that are the identity in a neighborhood of~$p$.  If $h\in \Stab_H(p)$, we let $(h)_p$ denote its image in~$(H)_p$.

\begin{lemma}\label{lem:VirtuallyCyclic}
If $G\leq \Aut(\T_d)$ is a contracting self-similar group and $p\in \C_d$ is a rational point, then $(V_dG)_p$ is  virtually cyclic.
\end{lemma}
\begin{proof}
Write $p=\alpha \overline{\beta}$ for some $\alpha,\beta\in X_d^+$, where $\beta$ is not a power of a shorter word.  By Lemma~\ref{lem:MakeElements}, there exists an $f\in V_dG$ such that
\[
f(\alpha\psi) = \alpha\beta\psi
\]
for all $\psi\in X_d^\omega$.  We claim that the cyclic group $\langle(f)_p\rangle$ generated by $(f)_p$ has finite index in $(V_dG)_p$.

Let $\mathcal{N}$ be the nucleus for $G$, and consider the mapping $g\mapsto g|_\beta$ on $\mathcal{N}$.  Since $\mathcal{N}$ is finite, every element of $\mathcal{N}$ is periodic or pre-periodic under this mapping.  Let $\mathcal{N}_\beta \subseteq \mathcal{N}$ denote the set of periodic points, and let $M$ be the least common multiple of their periods.

Now, if $h$ is any element of $V_dG$ that fixes $p$, then for all sufficiently large $i$ the cone  $\alpha\beta^i\C_d$ is regular for $h$ and the local action $h|_{\alpha\beta^i}$ lies in the nucleus $\mathcal{N}$.   Indeed, $h|_{\alpha\beta^i}$ must eventually lie in $\mathcal{N}_\beta$, and therefore the sequence of local actions $h|_{\alpha\beta^{Mi}}$ is eventually constant.  This defines a function $\mathrm{Stab}_{V_dG}(p) \to \mathcal{N}_\beta$, which descends to a well-defined function $(V_dG)_p\to \mathcal{N}_\beta$. 
 We claim that each fiber lies in a single right coset of $\langle(f)_p\rangle$.

Let $h$ and $h'$ be any two elements of $V_dG$ that fix $p$, and suppose that $h|_{\alpha\beta^{Mi}}=h'|_{\alpha\beta^{Mi}}$ for some~$i$.  Let $\zeta\C_d = h(\alpha\beta^{Mi}\C_d)$ and $\zeta'\C_d = h'(\alpha\beta^{Mi}\C_d)$.  Since $h$ and $h'$ fix $p$, both these cones must contain $p$, so $\zeta$ and $\zeta'$ are finite prefixes of $\alpha\overline{\beta}$.  Increasing $i$ if necessary, we may assume that both $\zeta$ and $\zeta'$ have $\alpha$ as a prefix.  But since
\[
\alpha\overline{\beta} = h(\alpha\overline{\beta}) = \zeta\cdot h|_{\alpha\beta^{Mi}}(\overline{\beta})
\]
and
\[
\alpha\overline{\beta} = h'(\alpha\overline{\beta}) = \zeta'\cdot h'|_{\alpha\beta^{Mi}}(\overline{\beta}) = \zeta'\cdot h|_{\alpha\beta^{Mi}}(\overline{\beta}) 
\]
and $\beta$ is not a power of any shorter word, the only possibility is that $\zeta = \alpha\beta^j\gamma$ and $\zeta'=\alpha\beta^{j'}\gamma$ for some $j,j'\geq 0$ and some proper prefix $\gamma$ of $\beta$.  Without loss of generality, we may assume that $j\leq j'$.  Then $h'$ agrees with $f^{j'-j}h$ on $\alpha\beta^{Mi}\C_d$, so $(h')_p=(f)_p^{j'-j}(h)_p$, and therefore $(h)$ and $(h')$ lie in the same right coset of $\langle(f)_p\rangle$.  Thus every fiber under the mapping $(V_dG)\to \mathcal{N}_\beta$ defined above lies in a single right coset of $\langle(f)_p\rangle$, and since $\mathcal{N}_\beta$ is finite we conclude that $\langle(f)_p\rangle$ has finite index in~$(V_dG)_p$.
\end{proof}

\begin{remark}
If $V_d$ is the Higman--Thompson group, it is well-known that $(V_d)_p$ is infinite cyclic for any rational point $p$, and the lemma above shows that $(V_d)_p$ has finite index in $(V_dG)_p$.  Using the methods of the proof above, it is also easy to show that $(G)_p$ is finite for any self-similar group~$G$ and any rational point~$p$.

However, it is not true in general that $(V_dG)_p$ is generated by $(V_d)_p$ and $(G)_p$.  For example, let $G\leq \Aut(\T_2)$ be the two-element self-similar group $\{1,a\}$, where $a$ is the reflection of $\mathcal{T}_2$ that switches all $0$'s and $1$'s in an infinite binary sequence, 
and let $p$ be the point $\overline{01}$.  Then $(G)_p$ is trivial but $(V_2)_p$ has index two in $(V_2G)_p$.  In particular, if $h$ is any element of $V_2G$ such that $h(0\C_2)=01\C_2$ and $h|_0=a$, then $(h)_p$ lies in $(V_2G)_p$ but not $(V_2)_p$.
\end{remark}

The following proposition completes the proof of Theorem~\ref{thm:MainTheorem}.  It follows the arguments given by the authors and James Hyde in \cite{BHM1}, as well as some unpublished simplifications partially due to James Hyde.

\begin{proposition}
If $G\leq \Aut(\T_d)$ is a contracting self-similar group and $S$ is a finite set of rational points in~$\C_d$, then the stabilizer of $S$ in  $V_dG$ is finitely generated.
\end{proposition}
\begin{proof}
Let $H=V_dG$ and let $\Fix_{H}(S) = \bigcap_{p\in S} \Stab_{H}(p)$ be the pointwise stabilizer of $S$.  This has finite index in $\Stab_{H}(S)$, so it suffices to prove that $\Fix_{H}(S)$ is finitely generated.

Let $\pi\colon \Fix_{H}(S) \to \prod_{p\in S} (H)_p$ be the homomorphism whose components are the quotient homomorphisms $\Stab_{H}(p)\to (H)_p$ for $p\in S$.   Since each $(H)_p$ is virtually cyclic by Lemma~\ref{lem:VirtuallyCyclic}, the image of $\pi$ is finitely generated,
so it suffices to prove that $\ker(\pi)=\bigcap_{p\in S} \Fix_{H}^0(p)$ is contained in some finitely generated subgroup of~$\Fix_{H}(S)$ (since then $\Fix_H(S)$ is generated by this subgroup and lifts of the generators of $\text{im}(\pi)$.) Note that this kernel consists precisely of those elements of $H$ that are the identity in a neighborhood of~$S$.

Write $S=\{p_1,\ldots,p_n\}$, where each $p_i=\alpha_i\overline{\beta_i}$ for some $\alpha_i,\beta_i\in X_d^+$.  Extending each $\alpha_i$ if necessary, we may assume that the cones $\alpha_1\C_d,\ldots,\alpha_n\C_d$ are pairwise disjoint and their union $E=\alpha_1\C_d\cup \cdots \cup \alpha_n\C_d$ is not all of~$\C_d$.  By Lemma~\ref{lem:MakeElements}, there exists an element $f\in H$ such that $f(\alpha_i\C_d)=\alpha_i\beta_i\C_d$ and $f|_{\alpha_i}=\mathrm{id}$ for each~$i$. Note that $f\in \Fix_H(S)$.  Furthermore, $E\supset f(E)\supset f^2(E)\supset \cdots$, and every open neighborhood of $S$ contains $f^k(E)$ for sufficiently large~$k$.  It follows that $\ker(\pi)=\bigcup_{i=0}^\infty f^i\Fix_H(E)f^{-i}$, so it suffices to prove that $\Fix_H(E)$ is contained in some finitely generated subgroup of~$\Fix_H(S)$.

Since $E$ is clopen, its complement $\C_d\setminus E$ can be written as a disjoint union of cones $\gamma_1\C_d,\ldots,\gamma_m\C_d$.  Let $k\in \N$ so that $m+k\equiv 1\;(\mathrm{mod}\;d-1)$, and choose pairwise disjoint cones $\gamma_{m+1}\C_d,\ldots,\gamma_{m+k}\C_d$ which are contained in $E$ but disjoint from~$S$.  Let $E'=\C_d\setminus (\gamma_1\C_d\cup \cdots \cup \gamma_{m+k}\C_d) = E\setminus (\gamma_{m+1}\C_d\cup \cdots \cup \gamma_{m+k}\C_d)$.  Then $\Fix_H(E)\subseteq \Fix_H(E')\subseteq \Fix_H(S)$, so it suffices to prove that $\Fix_H(E')$ is finitely generated.  Indeed, we will prove that $\Fix_H(E')\cong V_dG$, which is finitely generated by Nekrashevych's theorem~\cite{Nek2}.

Since $m+k\equiv 1\;(\mathrm{mod}\;d-1)$, there exists a partition of $\C_d$ into $m+k$ cones $\delta_1\C_d,\ldots,\delta_{m+k}\C_d$. Let $z\colon \C_d\setminus E'\to \C_d$ be the homeomorphism that maps each $\gamma_i\C_d$ to $\delta_i\C_d$ by the function $\gamma_i\psi\mapsto \delta_i\psi$. 

For each $h\in V_dG$, let $\varphi(h)$ be the element of $V_dG$ which is the identity on $E'$ and agrees with $z^{-1}hz$ on $\C_d\setminus E'$.  Then $\varphi\colon V_dG\to \Fix_H(E')$ is the desired isomorphism.  We conclude that $\Fix_H(E')$ is finitely generated, and therefore $\Fix_H(S)$ is finitely generated.
\end{proof}

\bigskip
\newcommand{\arXiv}[1]{\href{https://arxiv.org/abs/#1}{arXiv}}
\newcommand{\doi}[1]{\href{https://doi.org/#1}{Crossref\,}}
\bibliographystyle{plain}

\end{document}